\documentclass[]{article}
\usepackage{amssymb,amsmath,amsthm,mathdesign}
\usepackage{pdfsync}
\usepackage{color}

\newcommand{\R}{\mathbf{R}}

\newcommand {\E}{\mathrm{E}}

\renewcommand{\d}{\text{\rm d}}

\newcommand{\sG}{\mathcal{G}}

\newcommand{\sA}{\mathcal{A}}
\newcommand{\sE}{\mathcal{E}}

\newtheorem{stat}{Statement}[section]
\newtheorem{proposition}[stat]{Proposition}

\newtheorem{theorem}[stat]{Theorem}
\newtheorem{lemma}[stat]{Lemma}
\newtheorem{assumption}[stat]{Assumption}
\theoremstyle{definition}
\newtheorem{definition}[stat]{Definition}\newtheorem{remark}[stat]{Remark}

\numberwithin{equation}{section}

\numberwithin{equation}{section}

\newcommand{\RR}[1]{\mathbb{#1}}

\newcommand{\rd}{{\mathbb R^d}}

\def\R{{\mathbb R}}

\def\a{\alpha}

\def\E{{\mathbb E}}

\allowdisplaybreaks

\begin{document}

\title{Asymptotic properties of some space-time fractional stochastic equations}

\author{Mohammud Foodun\\
Loughborough University}
\author{Mohammud Foondun\\
Loughborough University\\
 and \\
   Erkan Nane\\
 Auburn University}


\maketitle


\begin{abstract}
 Consider non-linear time-fractional  stochastic heat type equations of the following type,
$$\partial^\beta_tu_t(x)=-\nu(-\Delta)^{\alpha/2} u_t(x)+I^{1-\beta}_t[\lambda \sigma(u)\stackrel{\cdot}{F}(t,x)]$$ in $(d+1)$ dimensions, where $\nu>0, \beta\in (0,1)$, $\alpha\in (0,2]$. The operator  $\partial^\beta_t$ is the Caputo fractional derivative while $-(-\Delta)^{\alpha/2} $ is the generator of an isotropic stable process and $I^{1-\beta}_t$ is the fractional integral operator. The forcing noise denoted by $\stackrel{\cdot}{F}(t,x)$ is a  Gaussian noise. And the multiplicative non-linearity $\sigma:\RR{R}\to\RR{R}$ is assumed to be globally Lipschitz continuous.

Under suitable conditions on the initial function, we study the asymptotic behaviour of the solution with respect to time and the parameter $\lambda$. In particular, our results are significant extensions of those  in \cite{foodun-liu-omaba-2014}, \cite{nane-mijena-2014}, \cite{mijena-nane-2} and \cite{foondun-khoshnevisan-09}.  Along the way, we prove a number of interesting properties about the deterministic counterpart of the equation.
\end{abstract}


\section{Introduction and main results.}
\subsection{Background material.}
Recently, there has been an increased interest in fractional calculus. This is because, time fractional operators are proving to be very useful for modelling purposes. For example, while the classical heat equation $\partial_tu_t(x)=\Delta u_t(x)$, used for modelling heat diffusion in homogeneous media, the fractional heat equation $\partial^\beta_tu_t(x)=\Delta u_t(x)$ are used to describe heat propagation in inhomogeneous media. It is known that as opposed to the classical heat equation, this equation is known to exhibit sub diffusive behaviour and are related with anomalous diffusions or diffusions in non-homogeneous media, with random fractal structures; see, for instance, \cite{meerschaert-nane-xiao}. The main aim of this paper is study a class of stochastic fractional heat equations. In particular, it will become clear how this sub diffusive feature affects other properties of the solution.

Stochastic partial differential equations (SPDE) have been studied  in  mathematics, and  in many disciplines that include  statistical mechanics, theoretical physics,  theoretical neuroscience, theory of complex chemical reactions, fluid dynamics, hydrology,  and mathematical finance; see, for example, Khoshnevisan \cite{khoshnevisan-cbms} for an extensive list of references. The area of SPDEs is interesting to mathematicians as it contains a lot of hard open problems. So far most of the work done on the stochastic heat equations have dealt with the usual time derivative, that is $\beta=1$. Its only recently that Mijena and Nane has introduced time fractional SPDEs in \cite{nane-mijena-2014}. These types of time fractional stochastic heat type equations are  attractive models that can be used to model phenomenon with random effects with thermal memory. In another paper \cite{mijena-nane-2} they have proved  exponential growth of solutions of time fractional SPDEs--intermittency-- under the assumption that the initial function is bounded from below.
A related  class of  time-fractional  SPDE was studied by  Karczewska \cite{karczewska}, Chen  et al. \cite{chen-kim-kim-2014}, and Baeumer et al \cite{baeumer-Geissert-Kovacs}. They have proved regularity of the solutions to the time-fractional parabolic type SPDEs using cylindrical Brownain motion in Banach spaces  in the sense of  \cite{daPrato-Zabczyk}. For a comparison of the two approaches to SPDE's see the paper by Dalang and Quer-Sardanyons \cite{Dalang-Quer-Sardanyons}.

A possible {\bf Physical explanation} of  time fractional SPDEs  is given in \cite{chen-kim-kim-2014}.
The time-fractional SPDEs studied in this paper may arise naturally by considering the heat equation in a material with thermal memory.

Before we describe our equations with more care, we provide some heuristics. Consider the following fractional equation,
\begin{equation*}
\begin{aligned}
\partial^\beta_tu_t(x)&=-\nu(-\Delta)^{\alpha/2} u_t(x)\\
\end{aligned}
\end{equation*}
with $\beta\in (0,1)$  and $\partial^\beta_t$ is the Caputo fractional derivative which first appeared in \cite{Caputo} and is defined by
\begin{equation}\label{CaputoDef}
\partial^\beta_t u_t(x)=\frac{1}{\Gamma(1-\beta)}\int_0^t \partial_r
u_r(x)\frac{\d r}{(t-r)^\beta} .
\end{equation}
If $u_0(x)$ denotes the initial condition to the above equation, then the solution can be written as
\begin{equation*}
u_t(x)=\int_\rd G_{t}(x-y)u_0(y)\d y.
\end{equation*}
$G_t(x)$ is the time-fractional heat kernel, which we will analyse a bit more later.  Let us now look at
 \begin{equation}\label{tfpde}
 \partial^\beta_tu_t(x)=-\nu(-\Delta)^{\alpha/2}u_t(x)+f(t,x),
 \end{equation}
with the same initial condition $u_0(x)$ and $f(t,x)$ is some nice function. We will make use of {\bf time fractional Duhamel's principle} \cite{umarov-06, umarov-12,Umarov-saydamatov} to get the correct version of \eqref{tfspde}.
Using the fractional Duhamel principle, the solution to \eqref{tfpde} is given by
$$u_t(x)=\int_\rd G_{t}(x-y)u_0(y)\d y+\int_0^t\int_\rd G_{t-r}(x-y)\partial^{1-\beta}_r f(r,y)\d y\d r.$$

We will remove the fractional derivative appearing in the second term of the above display. Define the fractional integral by
$$I^{\gamma}_tf(t):=\frac{1}{\Gamma(\gamma)} \int _0^t(t-\tau)^{\gamma-1}f(\tau)\d\tau.$$
For every  $\beta\in (0,1)$, and   $g\in L^\infty(\R_+)$ or $g\in C(\R_+)$
$$ \partial _t^\beta I^\beta_t g(t)=g(t).$$
We consider the time fractional PDE with a force given by $f(t,x)=I^{1-\beta}_tg(t,x)$, then by the Duhamel's principle, the  mild solution to
\eqref{tfpde} will be given by
$$u_t(x)=\int_{\rd} G_t(x-y)u_0(y)\d y+\int_0^t\int_{\rd} G_{t-r}(x-y) g(r,y)\d y\d r.$$
The reader can consult \cite{chen-kim-kim-2014} for more information. The first equation we will study in this paper is the following.
\begin{equation}\label{tfspde}
\begin{split}
\partial^\beta_t u_t(x)&=-\nu(-\Delta)^{\alpha/2} u_t(x)+I^{1-\beta}_t[\lambda\sigma(u_t(x))\stackrel{\cdot}{W}(t,x)],\, x\in \R^d,
 \end{split}
 \end{equation}
where the initial datum $u_0$ is a non-random measurable function. $\stackrel{\cdot}{W}(t,x)$ is a  space-time white noise with $x\in \R^d$, $I^{1-\beta}_t$ and $\sigma:\R\to\R$ is a globally Lipschitz function. $\lambda$ is a positive parameter called the ``level of noise".  We will make sense of the above equation using an idea Walsh \cite{walsh}. In light of the above discussion, a solution $u_t$ to the above equation will in fact be a solution to the following integral equation.
\begin{equation}\label{mild-sol-white}
u_t(x)=(\mathcal{G}u_0)_t(x)+\lambda\int_0^t\int_{\R^d}G_{t-s}(x-y)\sigma(u_s(y))W(\d y\,\d s),
\end{equation}
where
$$
(\mathcal{G}u_0)_t(x):=\int_\rd G_t(x-y)u_0(y)\d y.
$$

We now fix the parameters $\alpha$ and $\beta$. We will restrict $\beta\in (0,\,1)$. The dimension $d$ is related with $\alpha$ and $\beta$ via
\begin{equation*}
d<(2\wedge \beta^{-1})\alpha.
\end{equation*}
Note that when $\beta=1$, the equation reduces to the well known stochastic heat equation and the above restrict the problem to a one-dimensional one. This is the so called curse of dimensionality explored in \cite{foondun-khoshnevisan-Nualart-11}. We will require the following notion of "random-field" solution. We will need $d< 2\alpha$ while computing the $L^2-$norm of the heat kernel, while $d<\beta^{-1}\alpha$ is needed for an integrability condition needed for ensuring existence and uniqueness of the solution.
\begin{definition}
A random field $\{u_t(x), \ t\geq 0, x\in \rd\}$ is called a mild solution of \eqref{tfspde} if
\begin{enumerate}
\item $u_t(x)$ is jointly measurable in $t\geq 0$ and $x\in \rd$;
\item $\forall (t,x)\in [0,\infty)\times \rd$, $\int_0^t\int_{\R^d}G_{t-s}(x-y)\sigma(u_s(y))W(\d y\,\d s)$ is well-defined in $L^2(\Omega)$; by the Walsh-Dalang isometry this is the same as requiring $$
    \sup_{x\in \rd}\sup_{t>0}\E|u_t(x)|^2<\infty!
    $$
\item The following holds in $L^2(\Omega)$,
$$
u_t(x)=(\mathcal{G}u_0)_t(x)+\lambda\int_0^t\int_{\R^d}G_{t-s}(x-y)\sigma(u_s(y))W(\d y\,\d s).
$$
\end{enumerate}
\end{definition}

Next, we  introduce the  second class equation with space colored noise.
\begin{equation}\label{tfspde-colored}
\partial^\beta_t u_t(x)=-\nu(-\Delta)^{\alpha/2} u_t(x)+I_t^{1-\beta}[\lambda \sigma(u_t(x))\dot{F}(t,\,x)],\, x\in \R^d.
\end{equation}

The only difference with \eqref{tfspde} is that the noise term is now colored in space. All the other conditions are the same. We now briefly describe the noise.

$\dot{F}$ denotes the Gaussian colored noise satisfying the following property,
$$
\E[\dot{F}(t,x)\dot{F}(s,y)]=\delta_0(t-s)f(x,y).
$$
This can be interpreted more formally as

\begin{equation}\label{covariance-colored}
Cov \bigg(\int \phi \d F, \int \psi \d F\bigg)=\int_0^\infty\int _{\rd}\d x\int_{\rd}\d y\phi_s(x)\psi_s(y)f(x-y),
\end{equation}
where we use the notation $\int \phi \d F$ to denote the wiener integral of $\phi$ with respect to $F$, and  the right-most integral converges absolutely.

We will assume that the spatial correlation of the noise term is given by the following function for  $\gamma<d$,
$$
f(x,y):=\frac{1}{|x-y|^\gamma}.
$$

Following Walsh \cite{walsh}, we define the mild solution of \eqref{tfspde-colored}  as the predictable solution to the following integral equation

\begin{equation}\label{mild-sol-colored}
\begin{split}
u_t(x)&=(\mathcal{G}u_0)_t(x)+\lambda \int_\rd \int_0^t G_{t-s}(x-y)\sigma(u_s(y))F(\d s \d y).
\end{split}
\end{equation}
As before, we will look at random field solution, which is defined by \ref{mild-sol-colored}. We will also assume the following
\begin{equation*}
\gamma <\alpha\wedge d.
\end{equation*}
That we should have $\gamma<d$ follows from an integrability condition about the correlation function. We need $\gamma<\alpha$ which comes from an integrability condition needed for the existence and uniqueness of the solution.

{We now briefly give an outline of the paper. We adapt the methods of  proofs of the results in  \cite{foodun-liu-omaba-2014} with  many crucial nontrivial changes.  We state main results in the next section. We give some preliminary results in section 3, we prove a number of interesting properties of the heat kernel of the time fractional heat type partial differential equations that are essential to the proof of our main results. The proofs of the results in the space-time white noise are given in Section 4. In Section 5, we prove the main  results about the space colored noise equation, and the  continuity of the solution to the time fractional SPDEs with space colored noise.
Throughout the paper, we use the letter $C$ or $c$ with or without subscripts to denote
a constant whose value is not important and may vary from places to places.  If $x\in \R^d$, then $|x|$ will denote the euclidean norm of $x\in$, while when $A\subset \R^d$, $|A|$ will denote the Lebesgue measure of $A$.}

\subsection{Statement of main results.}
Before stating our main results precisely, we describe some of the conditions we need. The first condition is required for the existence-uniquess result as well as the upper bound on the second moment of the solution.

\begin{assumption}\label{existence-uniqueness}
\begin{itemize}
\item We assume that initial condition is a non-random bounded non-negative function $u_0:\R^d\rightarrow \R$.
\item We assume that $\sigma:\R\rightarrow \R$ is a globally Lipschitz function satisfying $\sigma(x)\leq L_\sigma|x|$ with $L_\sigma$ being a positive number.
\end{itemize}
\end{assumption}

The following condition is needed for lower bound on the second moment.

\begin{assumption}\label{lowerbound}
\begin{itemize}
\item We will assume that the initial function $u_0$ is non-negative on a set of positive measure.
\item The function $\sigma$ satisfies $\sigma(x)\geq l_\sigma|x|$ with $l_\sigma$ being a positive number.
\end{itemize}
\end{assumption}

Mijena and Nane  \cite[Theorem 2]{nane-mijena-2014} have essentially proved the next theorem. We give a new proof of  this theorem in this paper.

\begin{theorem}\label{white:upperbound}
Suppose that $d<(2\wedge \beta^{-1})\alpha$. Then under Assumption \ref{existence-uniqueness}, there exists a unique random-field solution to \eqref{tfspde} satisfying
\begin{equation*}
\sup_{x\in \R^d}\E|u_t(x)|^2\leq c_1e^{c_2\lambda^{\frac{2\alpha}{\alpha-d\beta}}t}\quad \text{for\,all}\quad t>0.
\end{equation*}
Here $c_1$ and $c_2$ are positive constants.
\end{theorem}
\begin{remark}
This theorem says that second moment grows at most exponentially.
While this has been known \cite{nane-mijena-2014}, the novelty here is that we give a
precise rate with respect to the parameter $
\lambda$. Theorem \ref{white:upperbound} implies that a random field solution exists when  $d<(2\wedge \beta^{-1})\alpha$.  It follows from this theorem that TFSPDEs in the case of space-time white noise  is that a random field solution exists in space dimension greater than 1 in some cases, in contrast to the  parabolic stochastic heat type equations, the case $\beta=1$. So in the case $\alpha=2, \beta<1/2$, a random field solution exists when $d=1,2,3$.
 When $\beta=1$ a random field solution exist only in spatial dimension $d=1$.

\end{remark}

The next theorem shows that under some additional condition, the second moment will have exponential growth. This greatly extends results of \cite{foondun-khoshnevisan-09}, \cite{foondun-khoshnevisan-13}, \cite{foodun-liu-omaba-2014} and \cite{Chen-le-2014}.

\begin{theorem}\label{white:lowerbound}
Suppose that the conditions of Theorem \ref{white:upperbound} are in force. Then under Assumption \ref{lowerbound}, there exists a $T>0$, such that
\begin{equation*}
\inf_{x\in B(0,\,t^{\beta/\alpha})}\E|u_t(x)|^2\geq c_3e^{c_4\lambda^{\frac{2\alpha}{\alpha-d\beta}}t}\quad \text{for\,all}\quad t>T.
\end{equation*}
Here $c_3$ and $c_4$ are positive constants.

\end{theorem}
 The lower bound in the previous theorem  is completely
new. Most of the results of these kinds have been derived from the renewal
theoretic ideas developed in \cite{foondun-khoshnevisan-09} and \cite{foondun-khoshnevisan-13}. The methods used in this article are
completely different. In particular, we make use of a localisation argument together
with heat kernel estimates for the  time fractional diffusion equation.

\begin{remark}
The two theorems above imply that, under some conditions, there exist some positive constants $a_5$ and $a_6$ such that,
 $$
a_5\lambda^{2\a/(\a-\beta d)}\leq \liminf_{t\to\infty}\frac{1}{t}\log \E|u_t(x)|^2\leq \limsup_{t\to\infty}\frac{1}{t}\log \E|u_t(x)|^2\leq a_6\lambda^{2\a/(\a-\beta d)} ,
 $$
 for any fixed $x\in \rd$.

 The exponential growth of the second moment of the solution have  been proved under the assumption that the initial function is bounded form below in \cite{mijena-nane-2}. This exponential growth property have been proved by \cite{foondun-khoshnevisan-09} when $\beta=1$ and $d=1$ when the initial function is also bounded from below. When $\beta=1$, and  the initial function satisfies the assumption \ref{lowerbound}, this was established by \cite{foodun-liu-omaba-2014}. Chen \cite{Chen-le-2014} has established intermittency of the solution of \eqref{tfspde} when $d=1$ and $\beta\in (0,1)$ and $\beta\in (1,2)$ with measure-valued initial data.

\end{remark}
We will need the following definition which we borrow from \cite{Khoshnevisan-kim}. Set
\begin{equation*}
\sE_t(\lambda):=\sqrt{\int_{\R^d}\E|u_t(x)|^2\,\d x}.
\end{equation*}

and define the nonlinear excitation index by
\begin{equation*}
e(t):=\lim_{\lambda \rightarrow \infty }\frac{\log \log \sE_t(\lambda)}{\log \lambda}.
\end{equation*}

The next theorem gives the rate of growth of the second moment with respect
to the parameter $\lambda$, which extends results in \cite{foodun-liu-omaba-2014}. We note that for time
$t$ large enough, this follows from the theorem above. But for small $t$, we need
to work a bit harder.
\begin{theorem}\label{thm:limit-lambda-white}
Fix $t>0$ and $x\in \R^d$, we then have
\begin{equation*}
\lim_{\lambda\rightarrow \infty} \frac{\log \log \E|u_t(x)|^2}{\log \lambda}=\frac{2\alpha}{\alpha-d\beta}.
\end{equation*}
Moreover, if the energy of the solution exists, then the excitation index, $e(t)$ is also equal to $\frac{2\alpha}{\alpha-d\beta}$.
\end{theorem}

Note that for the energy of the solution to exists, we need some assumption on the initial condition. One can always impose boundedness with compact support.

The following theorem is essentially Theorem 2 in \cite{nane-mijena-2014}. We only state it to compare the H\"older exponent with the excitation index. This shows that the relationship mentioned in \cite{foodun-liu-omaba-2014} holds for this equation as well: $\eta\leq 1/e(t)$. Hence showcasing a link between noise excitability and continuity of
the solution.

\begin{theorem}[\cite{nane-mijena-2014}]\label{thm:holder-exponent}
Let $\eta< (\a-\beta d)/2\a$ then for every $x\in \rd$, $\{u_t(x), t>0\}$, the solution to \eqref{tfspde} has H\"older continuous trajectories with exponent $\eta$.
\end{theorem}

All the above results were about the white noise driven equation. Our first result on space colored noise case reads as follows.
\begin{theorem}\label{thm-colored-noise} Under the Assumption \ref{existence-uniqueness},
there exists a unique random field solution $u_t$  of \eqref{tfspde-colored} whose second moment satisfies
\begin{equation*}
\sup_{x\in \R^d}\E|u_t(x)|^2\leq c_5 \exp(c_6\lambda^{2\alpha/(\alpha-\gamma\beta)}t)\quad\text{for\,all}\quad t>0.
\end{equation*}
Here the constants $c_5,  c_6$  are positive numbers.  If we impose the further requirement that Assumption \ref{lowerbound} holds, then there exists a $T>0$ such that
\begin{equation*}
\inf_{x\in B(0,\,t^{\beta/\alpha})}\E|u_t(x)|^2\geq c_7 \exp(c_8\lambda^{2\alpha/(\alpha-\gamma\beta)}t)\quad\text{for\,all}\quad t>T,
\end{equation*}
where $T$ and the constants $c_7, c_8$ are positive numbers.
\end{theorem}
\begin{remark}
Theorem \ref{thm-colored-noise} implies that  there exist some positive constants $c_9$ and  $c_{10}$  such that
 $$
c_9\lambda^{2\a/(\a-\beta \gamma)}\leq \liminf_{t\to\infty}\frac{1}{t}\log \E|u_t(x)|^2\leq \limsup_{t\to\infty}\frac{1}{t}\log \E|u_t(x)|^2\leq c_{10}\lambda^{2\a/(\a-\beta \gamma)} ,
 $$
 for any fixed $x\in \rd$.
\end{remark}

\begin{theorem}\label{excitation-colored}
Fix $t>0$ and $x\in \R^d$, we then have
\begin{equation*}
\lim_{\lambda\rightarrow \infty} \frac{\log \log \E|u_t(x)|^2}{\log \lambda}=\frac{2\alpha}{\alpha-\gamma\beta}.
\end{equation*}
Moreover, if the energy of the solution exists, then the excitation index, $e(t)$ is also equal to $\frac{2\alpha}{\alpha-\gamma\beta}$.
\end{theorem}

We now give a relationship between the excitation index of \eqref{tfspde-colored}
\begin{theorem}\label{thm:holder-exponent}
Let $\eta< (\a-\beta \gamma)/2\a$ then for every $x\in \rd$, $\{u_t(x), t>0\}$, the solution to \eqref{tfspde-colored} has H\"older continuous trajectories with exponent $\eta$.
\end{theorem}
While the general strategy of proof is the same as that used in \cite{foodun-liu-omaba-2014}, here we  develop some new important tools.  For example, we need analyse the heat kernel and prove some relevant estimates. In \cite{foodun-liu-omaba-2014}, this step was relatively straightforward. But here the lack of semigroup property makes it that we need to work much harder. To address this, we heavily rely on subordination. This insight, absent in \cite{Chen-le-2014} makes it that we are able to vastly generalise the results of that paper.
Another key tool is showing that with time, $(\mathcal{G}u_0)_t(x)$ decays at most like the inverse of a polynomial. This also requires techniques based on subordination.

\section{Preliminaries.}
As mentioned in the introduction, the behaviour heat kernel $G_t(x)$ will play an important role. This section will mainly be devoted to estimates involving this quantity.  We start by giving a stochastic representation of this kernel.  Let $X_t$ denote a symmetric $\alpha$ stable process with density function denoted by $p(t,\,x)$. This is characterized through the Fourier transform which is given by

\begin{equation}\label{Eq:F_pX}
\widehat{p(t,\,\xi)}=e^{-t\nu|\xi|^\alpha}.
\end{equation}

Let $D=\{D_r,\,r\ge0\}$ denote a $\beta$-stable subordinator and $E_t$ be its first passage time. It is known that the density of the time changed process $X_{E_t}$ is given by the $G_t(x)$. By conditioning, we have

\begin{equation}\label{Eq:Green1}
G_t(x)=\int_{0}^\infty p(s,\,x) f_{E_t}(s)\d s,
\end{equation}
where
\begin{equation}\label{Etdens0}
f_{E_t}(x)=t\beta^{-1}x^{-1-1/\beta}g_\beta(tx^{-1/\beta}),
\end{equation}
where $g_\beta(\cdot)$ is the density function of $D_1.$ and is infinitely differentiable on the entire real line, with $g_\beta(u)=0$ for $u\le 0$. Moreover,
\begin{equation}\label{Eq:gbeta0}
g_\beta(u)\sim K(\beta/u)^{(1-\beta/2)/(1-\beta)}\exp\{-|1-\beta|(u/\beta)^{\beta/(\beta-1)}\}\quad\mbox{as}\,\, u\to0+,
\end{equation}
and
\begin{equation}\label{Eq:gbetainf}
g_\beta(u)\sim\frac{\beta}{\Gamma(1-\beta)}u^{-\beta-1} \quad\mbox{as}\,\, u\to\infty.
\end{equation}

While the above expressions will be very important, we will also need the Fourier transform of $G_t(x)$.

\begin{equation*}
\hat{G_t}(\xi)=E_\beta(-\nu|\xi|^\alpha t^{\beta}),
\end{equation*}
where
\begin{equation}\label{ML-function}
E_\beta(x) = \sum_{k=0}^\infty\frac{x^k}{\Gamma(1+\beta k)},
\end{equation}
and
\begin{equation}\label{uniformbound}
 \frac{1}{1 + \Gamma(1-\beta)x}\leq E_{\beta}(-x)\leq \frac{1}{1+\Gamma(1+\beta)^{-1}x} \ \  \ \text{for}\ x>0,
 \end{equation}
 see, for example, \cite[Theorem 4]{simon}.
Even though, we will be mainly using the representation given by \eqref{Eq:Green1}, we also have another explicit description of the heat kernel.

Using the convention $\sim$ to denote the Laplace transform and $\ast$ the Fourier transform we get
\begin{equation}
\tilde{G}^\ast_t(x) = \frac{\lambda^{\beta-1}}{\lambda^{\beta} +\nu |\xi|^\alpha}.
\end{equation}
Inverting the Laplace transform yields
\begin{equation}\label{fouriertransformofG}
G^\ast_t(\xi) = E_\beta(-\nu|\xi|^\alpha t^\beta),
\end{equation}
where
\begin{equation}\label{ML-function}
E_\beta(x) = \sum_{k=0}^\infty\frac{x^k}{\Gamma(1+\beta k)},
\end{equation}
 is the Mittag-Leffler function. In order to invert the Fourier transform, we will make use of the integral \cite[eq. 12.9]{haubold-mathai-saxena}
\begin{equation}
\int_0^\infty\cos(ks)E_{\beta,\a}(-as^\mu)ds = \frac{\pi}{k}H_{3,3}^{2,1}\bigg[\frac{k^\mu}{a}\bigg|^{(1,1), (\a,\beta), (1,\mu/2)}_{(1,\mu),(1,1),(1,\mu/2)}\bigg],\nonumber
\end{equation}
where $\mathcal{R}(\a)>0,\mathcal{\beta}>0,k>0,a>0, H_{p,q}^{m,n}$ is the H-function given in \cite[Definition 1.9.1, p. 55]{mathai} and the formula
\begin{equation}
\frac{1}{2\pi}\int_{-\infty}^\infty e^{-i\xi x}f(\xi)d\xi = \frac{1}{\pi}\int_0^\infty f(\xi)\cos(\xi x)d\xi.\nonumber
\end{equation}
Then this gives the function as
\begin{equation}\label{G-function}
G_t(x) = \frac{1}{|x|} H_{3,3}^{2,1}\bigg[\frac{|x|^\a}{\nu t^\beta}\bigg|^{(1,1), (1,\beta), (1,\a/2)}_{(1,\a),(1,1),(1,\a/2)}\bigg].
\end{equation}
Note that for $\a = 2$ using reduction formula for the H-function we have
\begin{equation}
G_t(x) = \frac{1}{|x|}H^{1,0}_{1,1}\bigg[\frac{|x|^2}{\nu t^\beta}\bigg|^{(1,\beta)}_{(1,2)}\bigg]
\end{equation}
Note  that for $\beta = 1$ it reduces to the Gaussian density
\begin{equation}
G_t(x) = \frac{1}{(4\nu\pi t)^{1/2}}\exp\left(-\frac{|x|^2}{4\nu t}\right).
\end{equation}

 We will need following properties of the heat kernel of stable process.
  \begin{itemize}
\item \begin{equation*}
p(t,\,x)=t^{-d/\alpha}p(1,\,t^{-1/\alpha}x).
\end{equation*}
\item \begin{equation*}
p(st,\,x)=s^{-d/\alpha}p(t,\,s^{-1/\alpha}x).
\end{equation*}
\item $p(t,\,x)\geq p(t,\,y)$whenever $|x|\leq |y|$.
 \item For $t$ large enough so that $p(t,\,0)\leq 1$ and $\tau\geq 2$, we have
 \begin{equation*}
p(t,\,\frac{1}{\tau}(x-y))\geq p(t,\,x)p(t,\,y).
\end{equation*}
\end{itemize}
All these properties, except the last one, are straightforward. They follow from scaling.  We therefore provide a quick proof of the last inequality.  Suppose that $t$ is large enough so that $p(t,\,0)\leq 1$. Now, we have that $\frac{|x-y|}{\tau}\leq \frac{2|x|}{\tau}\vee \frac{2|y|}{\tau}\leq |x|\vee|y|.$ Therefore by the monotonicity property of the heat kernel and the fact that time is large enough, we have
\begin{equation*}
\begin{aligned}
p(t,\,\frac{1}{\tau}(x-y))&\geq p(t,\,|x|\vee|y|)\\
&\geq p(t,\,|x|)\wedge p(t,\,|y|)\\
&\geq p(t,\,|x|)p(t,\,|y|).
\end{aligned}
\end{equation*}
\\

We will need the lower bound described in the following lemma. The upper bound is given for the sake of completeness and is true under the additional assumption that $\alpha>d$, a condition which we will not need in this paper.
\begin{lemma}\label{Lemma:G-bounds}
\begin{enumerate}
\item[(a)] There exists a positive constant $c_1$ such that for all $x\in \R^d$
\begin{equation*}
G_t(x)\geq c_1 \bigg(t^{-\beta d/\alpha}\wedge \frac{t^\beta}{|x|^{d+\alpha}}\bigg).
\end{equation*}
\item[(b)] If we further suppose that $\alpha>d$, then there exists a positive constant $c_2$ such that for all $x\in \R^d$
\begin{equation*}
G_t(x)\leq c_2 \bigg(t^{-\beta d/\alpha}\wedge \frac{t^\beta}{|x|^{d+\alpha}}\bigg).
\end{equation*}
\end{enumerate}
\end{lemma}
\begin{proof}
It is well known that the transition density $p(t,\,x)$ of any strictly stable process is given by
 \begin{equation}
 c_1\bigg(t^{-d/\alpha}\wedge \frac{t}{|x|^{d+\alpha}}\bigg)\leq p(t,\,x)\leq c_2\bigg(t^{-d/\alpha}\wedge \frac{t}{|x|^{d+\alpha}}\bigg),
 \end{equation}
 where $c_1$ and $c_2$ are positive constants. We have
\begin{equation*}
G_t(x)=\int_{0}^\infty p(s,\,x) f_{E_t}(s)\d s,
\end{equation*}
which after using \eqref{Etdens0} and an appropriate substitution gives the following
\begin{equation*}
G_t(x)=\int_{0}^\infty p((t/u)^\beta,\,x) g_\beta(u)\d u.
\end{equation*}

Suppose that $|x|\leq t^{\beta/\alpha}$ then  $t/|x|^{\alpha/\beta}\geq 1$. When we have $u\leq t/|x|^{\alpha/\beta}$, we can write
\begin{equation}
\begin{aligned}
\int_0^\infty p((t/u)^\beta,\,x)g_\beta(u)\d u&\geq c_5\int_0^{t/|x|^{\alpha/\beta}} (t/u)^{-\beta d/\alpha}g_\beta(u)\d u \\
&\geq c_6\int_0^{1} (t/u)^{-\beta d/\alpha}g_\beta(u)\d u\\
&=c_7t^{-\beta d/\alpha} \int_0^{1} u^{\beta d/\alpha}g_\beta(u)du.\end{aligned}
\end{equation}
Since the integral appearing in the right hand side of the above display is finite, we have $G_t(x)\geq c_8t^{-\beta d/\alpha}$ whenever $|x|\leq t^{\beta/\alpha}$. We now look at the case $|x|\geq t^{\beta/\alpha}$.
\begin{equation}
\begin{aligned}
\int_0^\infty p((t/u)^\beta,\,x)g_\beta(u)\d u&\geq \int_{t/|x|^{\alpha/\beta}}^\infty c_9\frac{(t/u)^{\beta}}{|x|^{d+\alpha}}g_\beta(u)\d u\\
&\geq c_{10}\frac{t^\beta}{|x|^{d+\alpha}}\int_1^{\infty} (u)^{-\beta}g_\beta(u)\d u\\
& \geq \frac{c_{11}t^\beta}{|x|^{d+\alpha}},
\end{aligned}
\end{equation}
where we have used the fact that $\int_1^{\infty} (u)^{-\beta}g_\beta(u)du$ is a positive finite constant to come up with the last line.

We now use the fact that $p((t/u)^\beta,\,x)\leq c_1 \frac{u^{\beta d/\alpha}}{t^{\beta d/\alpha}}$, we have
\begin{equation*}
\begin{aligned}
G_t(x)&\leq c_1 \int_{0}^\infty \frac{u^{\beta d/\alpha}}{t^{\beta d/\alpha}} g_\beta(u)\d u\\
&=\frac{c_1}{t^{\beta d/\alpha}}\int_{0}^\infty u^{\beta d/\alpha} g_\beta(u)\d u.
\end{aligned}
\end{equation*}
The inequality on the right hand side is bounded only if $\alpha>d$. This follows from the fact that for large $u$, $g_\beta(u)$ behaves like $u^{-\beta-1}$. So we have $G_t(x)\leq \frac{c_2}{t^{\beta d/\alpha}}$. Similarly, we can use $p((t/u)^\beta,\,x)\leq \frac{c_3 t^{\beta}}{u^\beta|x|^{d+\alpha}}$, to write
\begin{equation*}
\begin{aligned}
G_t(x)&\leq c_3 \int_{0}^\infty \frac{t^{\beta}}{u^\beta|x|^{d+\alpha}}g_\beta(u)\d u\\
&=\frac{c_3t^{\beta}}{|x|^{d+\alpha}}\int_{0}^\infty u^{-\beta} g_\beta(u)\d u.
\end{aligned}
\end{equation*}
Since the integral appearing in the above display is finite, we have $G_t(x)\leq \frac{c_4t^{\beta}}{|x|^{d+\alpha}}$.  We therefore have
\begin{equation*}
G_t(x)\leq c_5\bigg(t^{-d\beta/\alpha}\wedge \frac{t^\beta}{|x|^{d+\alpha}}\bigg).
\end{equation*}
\end{proof}
\begin{remark}
When $\alpha\leq d$, then the function $G_t(x)$ is not well defined everywhere. But its representation in terms of H functions, one can show that $x=0$ is the only point where it is undefined. We won't use the pointwise upper bound. The lower bound is trivially true when $x=0$.
\end{remark}
The $L^2$-norm of the heat kernel can be calculated as follows.
\begin{lemma}\label{Lem:Green1} Suppose that $d < 2\alpha$, then
\begin{equation}\label{Eq:Greenint}
\int_{{\R^d}}G^2_t(x)\d x  =C^\ast t^{-\beta d/\alpha},
\end{equation}
where the constant $C^{\ast}$ is given by
\begin{equation*}
C^\ast = \frac{(\nu )^{-d/\alpha}2\pi^{d/2}}{\alpha\Gamma(\frac d2)}\frac{1}{(2\pi)^d}\int_0^\infty z^{d/\alpha-1} (E_\beta(-z))^2 \d z.
\end{equation*}
\end{lemma}

\begin{proof}
Using Plancherel theorem and \eqref{fouriertransformofG}, we have
\begin{eqnarray}
\int_{\R^d}|G_t(x)|^2 \d x &=& \frac{1}{(2\pi)^d}\int_{\R^d}|\hat{G_t}(\xi)|^2 \d \xi= \frac{1}{(2\pi)^d}\int_{\R^d}|E_\beta(-\nu|\xi|^\alpha t^\beta)|^2 \d \xi\nonumber\\
&=&\frac{2\pi^{d/2}}{\Gamma(\frac d2)}\frac{1}{(2\pi)^d}\int_0^\infty r^{d-1} (E_\beta(-\nu r^\alpha t^\beta))^2 \d r.\label{square-kernel}\\
&=&\frac{(\nu t^\beta)^{-d/\alpha}2\pi^{d/2}}{\alpha\Gamma(\frac d2)}\frac{1}{(2\pi)^d}\int_0^\infty z^{d/\alpha-1} (E_\beta(-z))^2 \d z.
\end{eqnarray}
To finish the proof, we need to show that the integral on the right hand side of the above display is bounded. We use equation \eqref{uniformbound} to get
\begin{eqnarray}\label{uniformboundforE}
\int_{0}^\infty\frac{z^{d/\alpha-1} }{(1+\Gamma(1-\beta)z)^2} \d r&\leq&\int_0^\infty z^{d/\alpha-1} (E_\beta(-z))^2 \d z\nonumber\\
&\leq&\int_{0}^\infty\frac{z^{d/\alpha-1} }{(1+\Gamma(1+\beta)^{-1}z)^2}\d z.
\end{eqnarray}
Hence  $\int_0^\infty z^{d/\alpha-1} (E_\beta(-z))^2 \d z<\infty$ if and only if $d<2\alpha$.
\end{proof}
Recall the Fourier transform of the heat kernel
\begin{equation}
G^\ast_t(\xi) = E_\beta(-\nu|\xi|^\alpha t^\beta).
\end{equation}
We will use this to prove the following.
\begin{lemma}\label{Lem:Green1-gamma} For $\gamma < 2\a,$
\begin{equation}\label{Eq:Greenint-gamma}
\int_{{\R^d}}[\hat{G_t}(\xi)]^2\frac{1}{|\xi|^{d-\gamma}} \d \xi  =C^\ast_1 t^{-\beta\gamma/\a},
\end{equation}
where $C^\ast _1= \frac{(\nu )^{-\gamma/\a}2\pi^{d/2}}{\a\Gamma(\frac d2)}\frac{1}{(2\pi)^d}\int_0^\infty z^{\gamma/\a-1} (E_\beta(-z))^2 \d z.$
\end{lemma}

\begin{proof}
We have

\begin{eqnarray}
\int_{{\R^d}}[G^\ast_t(\xi)]^2\frac{1}{|\xi|^{d-\gamma}} \d \xi &= &\int_{\rd}|E_\beta(-\nu|\xi|^\a t^\beta)|^2 \frac{1}{|\xi|^{d-\gamma}} \d \xi\nonumber\\
&=&\frac{2\pi^{d/2}}{\Gamma(\frac d2)}\int_0^\infty r^{d-1} (E_\beta(-\nu r^\a t^\beta))^2 \frac{1}{r^{d-\gamma}}\d r.\label{square-kernel-gamma}\\
&=&\frac{(\nu t^\beta)^{-\gamma/\a}2\pi^{d/2}}{\a\Gamma(\frac d2)}\frac{1}{(2\pi)^d}\int_0^\infty z^{\gamma/\a-1} (E_\beta(-z))^2 \d z.\nonumber
\end{eqnarray}
We used the integration in  polar coordinates for radially symmetric function in  the last equation above. Now using equation \eqref{uniformbound} we get
\begin{eqnarray}\label{uniformboundforE-gamma}
\int_{0}^\infty\frac{z^{\gamma/\a-1} }{(1+\Gamma(1-\beta)z)^2} \d r&\leq&\int_0^\infty z^{\gamma/\a-1} (E_\beta(-z))^2 \d z\nonumber\\
&\leq&\int_{0}^\infty\frac{z^{\gamma/\a-1} }{(1+\Gamma(1+\beta)^{-1}z)^2}\d z.
\end{eqnarray}
Hence  $\int_0^\infty z^{\gamma/\a-1} (E_\beta(-z))^2 \d z<\infty$ if and only if $\gamma<2\a$. In this case
 the upper bound in equation \eqref{uniformboundforE-gamma} is
$$\int_0^\infty \frac{z^{\gamma/\a-1} }{(1+\Gamma(1+\beta)^{-1}z)^2} \d z = \frac{\text{B}(\gamma/\a, 2-\gamma/\a)}{\Gamma(1+\beta)^{-\gamma/\a}},$$
where $\text{B}(\gamma/\a, 2-\gamma/\a)$ is a Beta function.
\end{proof}

\begin{remark}For $\gamma<2\a$,
\begin{equation}\label{Eq:Greenint2}
\frac{\text{B}(\gamma/\a, 2-\gamma/\a)}{\Gamma(1-\beta)^{\gamma/\a}} \leq\int_{0}^\infty z^{\gamma/\a -1}(E_\beta(-z))^2\d z \leq \frac{\text{B}(\gamma/\a, 2-\gamma/\a)}{\Gamma(1+\beta)^{-\gamma/\a}}.\nonumber
\end{equation}
\end{remark}

We have the following estimate which will be useful for establishing temporal continuity property of the solution of \eqref{tfspde-colored}.
\begin{proposition}\label{prop:temporal-increment-bound}
Let $\gamma<\min \{2, \beta^{-1}\}\alpha$ and $h\in (0,1)$, we then have
$$
\int_0^t\int_{\rd}|\hat{G}_{t-s+h}(\xi)-\hat{G}_{t-s}(\xi)|^2\frac{1}{|\xi|^{d-\gamma}}\ \d \xi \d s\leq c_1 h^{1-\beta\gamma/\a}.
$$

\end{proposition}

\begin{proof}
The computation in Lemma \ref{Lem:Green1-gamma} we have
\begin{eqnarray}
&&\int_{\R^d}|\hat{G}_{t-s+h}(\xi)- \hat{G}_{t-s}(\xi)|^2\frac{1}{|\xi|^{d-\gamma}} \d\xi\nonumber\\
 &=&\int_{\R^d} (\hat{G}_{t-s+h}(\xi))^2 \frac{1}{|\xi|^{d-\gamma}}\d\xi + \int_{\R^d} (\hat{G}_{t-s}(\xi))^2 \frac{1}{|\xi|^{d-\gamma}}\d\xi \nonumber\\
 &-&2\int_{\R^d} \hat{G}_{t-s+h}(\xi)\hat{G}_{t-s}(\xi) \frac{1}{|\xi|^{d-\gamma}}\d\xi \nonumber\\
 &=& C^\ast_1 (t- s+h)^{-\beta \gamma/\a} + C^\ast_1 (t-s)^{-\beta \gamma/\a} -2\int_{\R^d} \hat{G}_{t-s+h}(\xi)\hat{G}_{t-s}(\xi) \frac{1}{|\xi|^{d-\gamma}}\d\xi.\nonumber
\end{eqnarray}

Using  integration in polar coordinates in $\rd$, and the fact that $h(z)=E_\beta(-z)$ is decreasing (since it is completely monotonic, i.e. $(-1)^nh^{(n)}(z)\geq 0$ for all $z>0$, $n=0,1,2,3,\cdots$), we get
\begin{eqnarray}
 &&2\int_{\R^d} \hat{G}_{t-s+h}(\xi)\hat{G}_{t-s}(\xi) \frac{1}{|\xi|^{d-\gamma}}\d\xi \nonumber\\
 &=&2\int_{\R^d}E_\beta(-\nu|\xi|^\alpha (t-s+h)^\beta) E_\beta(-\nu|\xi|^\alpha (t-s)^\beta) \frac{1}{|\xi|^{d-\gamma}}\d\xi \nonumber\\
 &\geq &2\int_{\R^d}E_\beta(-\nu|\xi|^\alpha (t-s+h)^\beta) E_\beta(-\nu|\xi|^\alpha (t-s+h)^\beta) \frac{1}{|\xi|^{d-\gamma}}\d\xi \nonumber\\
 &=& 2C^*_1(t-s+h)^{-\beta \gamma/\a}.\nonumber
\end{eqnarray}

Now integrating both sides wrt to $s$ from $0$ to $t$ we get
\begin{eqnarray}\label{asmtotitc}
&&\int_0^t\int_{\R^d}|\hat{G}_{t-s+h}(\xi)- \hat{G}_{t-s}(\xi)|^2\frac{1}{|\xi|^{d-\gamma}} \d\xi \d r\\
&\leq &\frac{-C^\ast_1 (h)^{1-\beta \gamma/\a}}{1-\beta \gamma/\a} + \frac{C^\ast_11 (t+h)^{1-\beta \gamma/\a}}{1-\beta \gamma/\a} + \frac{C^\ast_1 t^{1-\beta \gamma/\a}}{1-\beta \gamma/\a}\nonumber\\
&+&\frac{ 2C^\ast_1 (h)^{1-\beta \gamma/\a}}{1-\beta \gamma/\a}-\frac{ 2C^\ast_1(t+h)^{1-\beta d/\a}}{1-\beta d/\a}\nonumber\\
&=&\frac{C^\ast_1(h)^{1-\beta \gamma/\a}}{1-\beta \gamma/\a}-\frac{C^\ast_1 (t+h)^{1-\beta \gamma/\a}}{1-\beta \gamma/\a}+\frac{C^\ast_1 t^{1-\beta \gamma/\a}}{1-\beta \gamma/\a}\nonumber\\
&\leq &\frac{C^\ast_1(h)^{1-\beta \gamma/\a}}{1-\beta \gamma/\a},\nonumber
\end{eqnarray}
the last inequality follows since $t<t'$.\\
\end{proof}

 \begin{lemma}\label{lemma:covariance-upper-bound}
Suppose that $\gamma<\alpha$, then there exists a constant $c_1$ such that for all $x,\,y\in \R^d$, we have
\begin{equation*}
\int_{\R^d}\int_{\R^d}G_t(x-w)G_t(y-z)f(z,\,w)\d w \d z\leq \frac{c_1}{t^{\gamma\beta/\alpha}}.
\end{equation*}
 \end{lemma}
\begin{proof}
We start by writing
\begin{equation*}
\begin{aligned}
\int_{\R^d}\int_{\R^d}&p(t,\,x-w)p(t',\,y-z)f(z,\,w)\d w \d z\\
&=\int_{\R^d}p(t+t',\,x-y+w)|w|^{-\gamma}\,\d w\\
&\leq \frac{c_2}{(t+t')^{\gamma/\alpha}}.
\end{aligned}
\end{equation*}
We use subordination again to write
\begin{equation*}
\begin{aligned}
\int_{\R^d}&\int_{\R^d}G_t(x-w)G_t(y-z)f(z,\,w)\d w \d z\\
&=\int_{\R^d\times \R^d} \int_0^\infty \int_0^\infty p(s,\,x-w)p(s',\,y-z)f_{X_t}(s)f_{X_t}(s')dsds'f(z,\,w)\d w \d z\\
&=\int_0^\infty \int_0^\infty\int_{\R^d\times \R^d} p(s,\,x-w)p(s',\,y-z)f(z,\,w)\d w \d zf_{X_t}(s)f_{X_t}(s')\d s\d s'\\
&=\int_0^\infty \int_0^\infty\int_{\R^d\times \R^d} p(s,\,x-w)p(s',\,y-z)f(z,\,w)\d w \d zf_{X_t}(s)f_{X_t}(s')\d s\d s'\\
&\leq \int_0^\infty \int_0^\infty\frac{c_2}{(s+s')^{\gamma/\alpha}}f_{X_t}(s)f_{X_t}(s')\d s\d s'\\
&\leq\int_0^\infty \int_0^\infty\frac{c_2}{s^{\gamma/\alpha}}f_{X_t}(s)f_{X_t}(s')\d s\d s'.
\end{aligned}
\end{equation*}
Recalling that $f_{X_t}(s')$ is a probability density, we can use a change of variable to see that the right hand side of the above display is bounded by
\begin{equation*}
\frac{c_3}{t^{\gamma\beta/\alpha}}\int_0^\infty u^{\gamma\beta/\alpha}g_\beta(u)\,\d u.
\end{equation*}
Since the above integral is finite, the result is proved.
\end{proof}

The next result gives the behaviour of non-random term for the mild formulation for the solution. For notational convenience, we set
\begin{equation*}
(\sG u)_t(x):=\int_{\R^d}G_t(x-y)u_0(y)\,\d y.
\end{equation*}
The proof will strongly rely on the representation given by \eqref{Eq:Green1} and we will also need
\begin{equation*}
(\tilde{\sG} u)_t(x):=\int_{\R^d}p(t,\,x-y)u_0(y)\,\d y,
\end{equation*}
where $p(t,\,x)$ is the heat kernel of the stable process. We will need the fact that for $t$ large enough, we have
 $(\tilde{\sG} u)_t(x)\geq c_1t^{-d/\alpha}$ for $x\in B(0,\,t^{1/\alpha})$.
We will prove this fact and a bit more in the following.  The proof heavily relies on the properties of $p(t,\,x)$ which we stated earlier in this section.
\begin{lemma}\label{stable-growth}
Then there exists a $t_0>0$ large enough such that for all $t>0$
 \begin{equation*}
 (\tilde{\sG} u)_{t+t_0}(x)\geq c_1t^{-d/\alpha},\quad\text{whenever}\quad x\in B(0,\,t^{1/\alpha}),
\end{equation*}
where $c_1$ is a positive constant. More generally, there exists a positive constant $\kappa>0$ such that for $s\leq t$ and $t\geq t_0$, we have
 \begin{equation*}
 (\tilde{\sG} u)_{s+t_0}(x)\geq c_2t^{-\kappa},\quad\text{whenever}\quad x\in B(0,\,t^{1/\alpha}).
\end{equation*}
$c_2$ is some positive constant.
\end{lemma}
\begin{proof}
We begin with the following observation about the heat kernel. Choose $t_0$ large enough so that $p(t_0,\,0)\leq 1$. We therefore have
\begin{equation*}
\begin{aligned}
p(t_0,\,x-y)&=p(t_0,\,2(x-y)/2)\\
&\geq p(t_0,\,2x)p(t_0,\,2y)\\
&=\frac{1}{2^d}p(t_0/2^\alpha,\,x)p(t_0,\,2y).
\end{aligned}
\end{equation*}
 This immediately gives
 \begin{equation*}
 \begin{aligned}
(\tilde{\sG} u)_{t_0}(x)&=\int_{\R^d}p(t,\,x-y)u_0(y)\,\d y\\
&\geq c_1p(t_0/2^\alpha,\,x)\int_{\R^d}p(t_0,\,2y)u_0(y)\,\d y.
 \end{aligned}
 \end{equation*}
 We now use the semigroup property to obtain
 \begin{eqnarray}
(\tilde{\sG} u)_{t+t_0}(x)&=&\int_{\R^d}p(t+t_0,\,x-y)u_0(y)\,\d y\nonumber\\
&=&\int_{\R^d}p(t,\,x-y )(\tilde{\sG} u)_{t_0}(y)\,\d y\nonumber\\
&\geq& c_2p(t+t_0/2,\,x),\label{ineq-semigroup}
 \end{eqnarray}
 This inequality shows that for any fixed $x$, $(\tilde{\sG} u)_{t+t_0}(x)$ decays as $t$ goes to infinity. It also shows that
 \begin{equation*}
 (\tilde{\sG} u)_{t+t_0}(x)\geq c_3t^{-d/\alpha},\quad\text{whenever}\quad |x|\leq t^{1/\alpha}.
\end{equation*}
This follows from the fact that $p(t+t_0/2,\,x)\geq c_4t^{-d/\alpha}$ if $|x|\leq t^{1/\alpha}$. The more general statement of the lemma needs a bit more work.
\begin{equation*}
\begin{aligned}
(\tilde{\sG} u)_{s+t_0}(x)&\geq c_2p(s+t_0/2,\,x)\\
&\geq c_3\left(\frac{t_0}{2s+t_0}\right)^{d/\alpha}p(t_0,\,x)\\
&\geq c_3\left(\frac{t_0}{2s+t_0}\right)^{d/\alpha}p(t_0,\,t^{1/\alpha}).
\end{aligned}
\end{equation*}
Since we are interested in the case when $s\leq t$ and $t\geq t_0$, the right hand side can be bounded as follows
\begin{equation*}
(\tilde{\sG} u)_{s+t_0}(x)\geq c_4\left(\frac{t_0}{2t+t_0}\right)^{d/\alpha}\frac{t_0}{t^{d/\alpha+1}}.
\end{equation*}
The second inequality in the statement of the lemma follows from the above.
\end{proof}

\begin{lemma}\label{lemma:lower-bound-frac-diffusion}
There exists a  $t_0>0$ and a constant $c_1$ such that for all $t>t_0$ and all $x\in B(0,\,t^{\beta/\alpha})$, we have
 \begin{equation*}
(\sG u)_{s+t}(x)\geq \frac{c_1}{t^{\beta \kappa}}\quad\text{for all}\quad s\leq t.
 \end{equation*}
\end{lemma}

\begin{proof}
We start off by writing
\begin{equation*}
\begin{aligned}
(\sG u)_t(x)&=\int_{\R^d}G_t(x-y)u_0(y)\,\d  y\\
&=\int_{\R^d}\int_0^\infty p(s,\,x-y)f_{X_t}(s)\,d s\,u_0(y)\d  y\\
&=\int_0^\infty (\tilde{\sG} u)_s(x)f_{X_t}(s)\,\d  s.
\end{aligned}
\end{equation*}

After the usual change of variable, we have
\begin{equation*}
(\sG u)_t(x)=\int_0^\infty (\tilde{\sG} u)_{(t/u)^\beta}(x)g_\beta(u)\,\d  u,
\end{equation*}
which immediately gives
\begin{equation*}
\begin{aligned}
(\sG u)_t(x)&\geq\int_0^1(\tilde{\sG} u)_{(t/u)^\beta}(x)g_\beta(u)\,\d u.
\end{aligned}
\end{equation*}
The above holds for any time $t$. In particular, we have
\begin{equation*}
\begin{aligned}
(\sG u)_{t+s}(x)&\geq\int_0^1(\tilde{\sG} u)_{((t+s)/u)^\beta}(x)g_\beta(u)\,\d u.
\end{aligned}
\end{equation*}
We now that note that $x\in B(0,\,t^{\beta/\alpha})$, so we have $x\in B(0,\,t^{\beta/\alpha}/u)$ and hence for $t$ large enough and $s\leq t$, we have $(\tilde{\sG} u)_{((s+t_0)/u)^\beta}(x)\geq \left(\frac{u}{t}\right)^{\beta\kappa}$ by the previous lemma.  Combining the above estimates, we have the result.\end{proof}

\begin{remark}\label{lower-bound-cauchy-problem}
The above is enough for the lower bound given in Theorem \ref{white:lowerbound} and the lower bound described in Theorem \ref{thm-colored-noise}. But we need an analogous result for the the noise excitability result which hold for all $t>0$.
Fix $\tilde{t}>0$ such that $p(t,\,0)\leq 1$ whenever $t\geq \tilde{t}$. For any fixed $t>0$, we choose $k$ large enough so that  $2^kt>\tilde{t}$. Set $t^*:=2^kt$ and $s=2^{-k}$.
\begin{equation*}
\begin{aligned}
p(t,\,x-y)&=p(st^*,\,x-y)\\
&=s^{-d/\alpha}p(t^*,\,s^{-1/\alpha}(x-y))\\
&=s^{-d/\alpha}p(t^*,\,\frac{s^{-1/\alpha}}{2}(2x-2y)).\\
\end{aligned}
\end{equation*}
For any fixed $t>0$, we choose $k$ large enough so that  $2^kt>\tilde{t}$.
\begin{equation*}
\begin{aligned}
p(t,\,x-y)&\geq s^{-d/\alpha}p(t^*,\,2s^{-1/\alpha}x)p(t^*,\,2s^{-1/\alpha}y)\\
&=2^{dk/\alpha}p(2^kt,\,2^{1+k/\alpha}x)p(2^kt,\,2^{1+k/\alpha}y).
\end{aligned}
\end{equation*}
Note that the above holds for any time $t$. We therefore have
\begin{equation*}
\begin{aligned}
(\tilde{\sG} u_0)_{t_0+s}(x)&=\int_{\R^d}p(t_0+s,\,x-y)u_0(y)\,\d y\\
&\geq 2^{dk/\alpha}p(2^k(t_0+s),\,2^{1+k/\alpha}x)\int_{\R^d}p(2^k(t_0+s),\,2^{1+k/\alpha}y)u_0(y)\,\d y.
\end{aligned}
\end{equation*}
We have that $t_0+s\geq t_0$. Therefore,
\begin{equation*}
p(2^k(t_0+s),\,2^{1+k/\alpha}x)\geq \left(\frac{t_0}{s+t_0}\right)^{d/\alpha}p(2^kt_0,\,2^{1+k/\alpha}x)
\end{equation*}
We thus have
\begin{equation*}
\begin{aligned}
(\tilde{\sG} u_0)_{t_0+s}(x)&\geq 2^{dk/\alpha}\left(\frac{t_0}{s+t_0}\right)^{2d/\alpha}p(2^kt_0,\,2^{1+k/\alpha}x)\int_{\R^d}p(2^kt_0,\,2^{1+k/\alpha}y)u_0(y)\,\d y.
\end{aligned}
\end{equation*}
So now since $|x|\leq t^{1/\alpha}$, we have
\begin{equation*}
(\tilde{\sG} u_0)_{t_0+s}(x)\geq c_1\left(\frac{1}{t_0+s}\right)^{2d/\alpha},
\end{equation*}
where the constant $c_1$ is dependent on $t_0$. We can now use  similar ideas as in the proof of the previous result to conclude that if $x\in B(0,\,t^{\beta/\alpha})$, we have
\begin{equation*}
(\sG u_0)_{t_0+s}(x)\geq c_2\left(\frac{1}{t_0+s}\right)^{2\beta d/\alpha}.
\end{equation*}
Since we have $s\leq t$, we have essentially found a lower bound for $(\sG u_0)_{t_0+s}(x)$; a bound which depends only on $t$. This holds for any $t_0>0$ and any $t>0$.
\end{remark}

We end this section with a few results from \cite{foodun-liu-omaba-2014}. These will  be useful for the proofs of our main results.
\begin{lemma}[Lemma 2.3 in \cite{foodun-liu-omaba-2014}]\label{lemma:exp-lower-bound}
Let $0<\rho<1$, then there exists a positive constant $c_1$ such that or all $b\geq (e/\rho)^\rho$,
$$
\sum_{j=0}^\infty\bigg(\frac{b}{j^\rho}\bigg)^j\geq \exp\bigg(c_1b^{1/\rho}\bigg).
$$
\end{lemma}
\begin{proposition}\label{prop:renewal-upper}[Proposition 2.5 in \cite{foodun-liu-omaba-2014}]
Let $\rho>0$ and suppose $f(t)$ is a locally integrable function satisfying
$$
f(t)\leq c_1+\kappa \int_0^t(t-s)^{\rho-1} f(s)\d s \ \ \mathrm{for \ all} \ \ t>0,
$$
where $c_1$ is some positive number. Then, we have
$$
f(t)\leq c_2\exp(c_3(\Gamma(\rho))^{1/\rho}\kappa^{1/\rho} t)\  \ \mathrm{for \ all} \ \ t>0,
$$ for some positive constants $c_2$ and $c_3$.
\end{proposition}
Also we give the following converse.
\begin{proposition}[Proposition 2.6 in \cite{foodun-liu-omaba-2014}]
Let $\rho>0$ and suppose $f(t)$ is nonnegative,  locally integrable function satisfying
$$
f(t)\geq c_1+\kappa \int_0^t(t-s)^{\rho-1} f(s)\d s \ \ \mathrm{for \ all} \ \ t>0,
$$
where $c_1$ is some positive number. Then, we have
$$
f(t)\geq c_2\exp(c_3(\Gamma(\rho))^{1/\rho}\kappa^{1/\rho} t)\  \ \mathrm{for \ all} \ \ t>0,
$$ for some positive constants $c_2$ and $c_3$.
\end{proposition}
\section{Proofs for the white noise case.}

\subsection{Proofs of Theorem \ref{white:upperbound}.}
\begin{proof}

We first show the existence of a unique solution. This follows from a standard Picard iteration; see \cite{walsh}, so we just briefly spell out the main ideas. For more information, see \cite{nane-mijena-2014}. Set
\begin{equation*}
u_t^{(0)}(x):=(\sG u_0)_t(x)
\end{equation*}
and
\begin{equation*}
u_t^{(n+1)}(x):=(\sG u_0)_t(x)+\lambda\int_0^t\int_{\R^d}G_{t-s}(x-y)\sigma(u^{(n)}_s(y))W(\d y\,\d s)\quad\text{for}\quad n\geq 0.
\end{equation*}
Define $D_n(t\,,x):=\E|u^{(n+1)}_t(x)-u^{(n)}_t(x)|^2$ and $H_n(t):=\sup_{x\in \R^d}D_n(t\,,x)$. We will prove the result for $t\in [0,\,T]$, where $T$ is some fixed number. We now use this notation together with Walsh's isometry and the assumption on $\sigma$ to write
\begin{equation*}
\begin{aligned}
D_n(t,\,x)&=\lambda^2\int_0^t\int_{\R^d}G^2_{t-s}(x-y)\E|\sigma(u^{(n)}_s(y))-\sigma(u^{(n-1)}_s(y))|^2\d y\,\d s\\
&\leq \lambda^2L_\sigma^2\int _0^tH_{n-1}(s)\int_{\R^d}G^2_{t-s}(x-y)\,\d y\,\d s\\
&\leq \lambda^2L_\sigma^2\int_0^T\frac{H_{n-1}(s)}{(t-s)^{d\beta/\alpha}}\,\d s
\end{aligned}
\end{equation*}
We therefore have
\begin{equation*}
H_{n}(t)\leq \lambda^2L_\sigma^2\int_0^T\frac{H_{n-1}(s)}{(t-s)^{d\beta/\alpha}}\,\d s.
\end{equation*}
We now note that the integral appearing on the right hand side of the above display is finite when $d<\alpha/\beta$. Hence, by Lemma 3.3 in Walsh \cite{walsh}, the series $\sum_{n=0}^\infty H^{\frac12}_n(t)$ converges uniformly on $[0,\,T].$ Therefore, the sequence $\{u_n\}$ converges  in $L^2$ and uniformly on $[0,\,T]\times\R^d$ and the limit satisfies (\ref{mild-sol-white}).  We can prove uniqueness in a similar way.
We now turn to the proof of the exponential bound. From Walsh's isometry, we have
\begin{equation*}
\E|u_t(x)|^2=|(\sG u_0)_t(x)|^2+\lambda^2\int_0^t\int_{\R^d}G^2_{t-s}(x-y)\E|\sigma(u_s(y))|^2\d y\,\d s.
\end{equation*}
Since we are assuming that the initial condition is bounded, we have that $|(\sG u_0)_t(x)|^2\leq c_1$ and the second term is bounded by
\begin{equation*}
\begin{aligned}
\lambda^2L_\sigma^2\int_0^t\int_{\R^d}&G^2_{t-s}(x-y)\E|u_s(y)|^2\d y\,\d s\\
&\leq c_1\lambda^2L_\sigma^2\int_0^t\frac{1}{(t-s)^{d\beta/\alpha}}\sup_{y\in \R^d}\E|u_s(y)|^2\d y\,\d s.
\end{aligned}
\end{equation*}
We therefore have
\begin{equation*}
\sup_{x\in \R^d}\E|u_s(x)|^2\leq c_1+c_2\lambda^2L_\sigma^2\int_0^t\frac{1}{(t-s)^{d\beta/\alpha}}\sup_{y\in \R^d}\E|u_s(y)|^2\,\d s.
\end{equation*}
The renewal inequality in Proposition \ref{prop:renewal-upper} with $\rho=(\alpha-d\beta)/\alpha$  proves the result.
\end{proof}

\subsection{Proof of Theorem \ref{white:lowerbound}.}
The proof of Theorem \ref{white:lowerbound} will rely on the following observation. From Walsh isometry, we have
\begin{equation*}
\E|u_t(x)|^2=|(\sG u_0)_t(x)|^2+\lambda^2\int_0^t\int_{\R^d}G^2_{t-s}(x-y)\E|\sigma(u_s(y))|^2\d y\,\d s.
\end{equation*}
For any fixed $t_0>0$, we use a change of variable and the fact that all the terms are non-negative to obtain
\begin{equation*}
\E|u_{t+t_0}(x)|^2\geq |(\sG u_0)_{t+t_0}(x)|^2+\lambda^2l_\sigma^2\int_0^t\int_{\R^d}G^2_{t-s}(x-y)\E|u_{s+t_0}(y)|^2\d y\,\d s.
\end{equation*}
Using the above relation again, we obtain
\begin{equation*}
\begin{aligned}
\E&|u_{t+t_0}(x)|^2\geq |(\sG u_0)_{t+t_0}(x)|^2\\
&+\lambda^2l_\sigma^2\int_0^t\int_{\R^d}G^2_{t-s}(x-y)|(\sG u_0)_{s+t_0}(x)|^2\d y\,\d s\\
&+\lambda^4l_\sigma^4\int_0^t\int_{\R^d}\int_0^s\int_{\R^d}G^2_{t-s}(x-y)G^2_{s-s_1}(y-z)\E|u_{s_1+t_0}(z)|^2\d z\,\d s_1\d y\,\d s.
\end{aligned}
\end{equation*}
Using the same procedure recursively, we obtain
\begin{equation}\label{infinite-sum}
\begin{aligned}
\E&|u_{t+t_0}(x)|^2\\
&\geq |(\sG u_0)_{t+t_0}(x)|^2\\
&+\sum_{k=1}^\infty \lambda^{2k}l_\sigma^{2k}\int_0^t\int_{\R^d}\int_0^{s_1}\int_{\R^d}\dots\int_0^{s_{k-1}}\int_{\R^d} |(\sG u_0)_{t_0+s_k}(z_k)|^2\\
&\prod_{i=1}^{k}G^2_{s_{i-1}-s_i}(z_{i-1},z_i)\,\d z_{k+1-i}\,\d s_{k+1-i},
\end{aligned}
\end{equation}
where we have used the convention that $s_0:=t$ and $z_0:=x$.  Let $x\in B(0,\,t^{\beta/\alpha})$ and $0\leq s\leq t$ and set
\begin{equation}\label{lowerbound:g}
(\sG u_0)_{t_0+s}(x)\geq g_t.
\end{equation}
The existence of such a function $g_t$ is guaranteed by Lemma \ref{lemma:lower-bound-frac-diffusion} and Remark \ref{lower-bound-cauchy-problem}. We can now use the above representation to prove the following result.

\begin{proposition}\label{whitenoiselowbd}
Fix $t_0>0$ such that for $t\geq 0$,
\begin{equation*}
\E|u_{t+t_0}(x)|^2\geq g_t^2\sum_{k=0}^\infty \left(\lambda^2l_\sigma^2c_1 \right)^k\left(\frac{t}{k}\right)^{k(\alpha-\beta d)/\alpha}\quad\text{for}\quad x\in B(0, t^{\beta/\alpha}),
\end{equation*}
where $c_1$ is a positive constant.
\end{proposition}
\begin{proof}
Our starting point is \eqref{infinite-sum}. Recall the notation introduced above,
\begin{equation*}
(\sG u_0)_{t_0+s_k}(z_k)\geq g_t,
\end{equation*}
whenever $z_k\in B(0, t^{\beta/\alpha})$ and $0\leq s_k\leq t$.  The infinite sum of the right of \eqref{infinite-sum} is thus bounded below by
\begin{equation*}
g_t^2\sum_{k=1}^\infty \lambda^{2k}l_\sigma^{2k}\int_0^t\int_{\R^d}\int_0^{s_1}\int_{\R^d}\dots\int_0^{s_{k-1}}\int_{B(0,\,t^{\beta/\alpha})}\prod_{i=1}^{k}G^2_{s_{i-1}-s_i}(z_{i-1},z_i)\,\d z_{k+1-i}\,\d s_{k+1-i}.
\end{equation*}
We now make a reduce the temporal domain of integration and make an appropriate change of variable to find a lower bound of the above display

\begin{equation*}
g_t^2\sum_{k=1}^\infty \lambda^{2k}l_\sigma^{2k}\int_0^{t/k}\int_{\R^d}\int_0^{t/k}\int_{\R^d}\dots\int_0^{t/k}\int_{B(0,\,t^{\beta/\alpha})}\prod_{i=1}^{k}G^2_{s_i}(z_{i-1},z_i)\,\d z_{k+1-i}\,\d s_{k+1-i}.
\end{equation*}
We will reduce the domain of the function
\begin{equation*}
\prod_{i=1}^{k}G^2_{s_i}(z_{i-1},z_i),
\end{equation*}
by choosing the points $z_i$ appropriately so that they are "not too far way".  We choose $z_1\in B(0,\,t^{\beta/\alpha})$ such that $|z_1-x_0|\leq s_1^{\beta/\alpha}$. In general, for $i=1,\cdots,k$, we choose $z_i\in B(z_{i-1},\,s_i^{\beta/\alpha})\cap B(0,\,t^{\beta/\alpha})$.  An immediate  consequence of this restriction is that

\begin{equation*}
\prod_{i=1}^{k}G^2_{s_i}(z_{i-1},z_i)\geq \prod_{i=1}^{k}\frac{c_1}{s_i^{2d\beta/\alpha}}.
\end{equation*}
Since the area of the set $B(z_{i-1},\,s_i^{\beta/\alpha})\cap B(0,\,t^{\beta/\alpha})$ is $c_2s_i^{d\beta/\alpha}$, we have
\begin{equation*}
\begin{aligned}
\int_0^{t/k}\int_{\R^d}\int_0^{t/k}\int_{\R^d}&\dots\int_0^{t/k}\int_{B(0,\,t^{\beta/\alpha})}\prod_{i=1}^{k}G^2_{s_i}(z_{i-1},z_i)\,\d z_{k+1-i}\,\d s_{k+1-i}\\
&\geq \int_0^{t/k}\cdots\int_0^{t/k} \frac{c_3^k}{s_i^{d\beta/\alpha}} \d s_{1}\cdots \d s_{k}\\
&=c^k_3\left(\frac{t}{k}\right)^{(\alpha-d\beta)k/\alpha}.
\end{aligned}
\end{equation*}

Putting all the estimate together we have
\begin{equation*}
\begin{aligned}
\E|&u_{t+t_0}(x)|^2\geq  g_t^2\sum_{k=0}^\infty \lambda^{2k}l_\sigma^{2k}c_4^k\left(\frac{t}{k}\right)^{(\alpha-d\beta)k/\alpha}.
\end{aligned}
\end{equation*}

\end{proof}
{\it Proof of Theorem \ref{white:lowerbound}.}
We make the important observation that $g_t$ decays no faster than polynomial. After a simple substitution and the use of Lemma \ref{lemma:exp-lower-bound}, the theorem is proved.
\qed
\begin{remark}
It should be noted that we do not need the full statement of Proposition \ref{whitenoiselowbd}.  All that we need is the statement when time is large.
\end{remark}

\subsection{Proof of Theorem \ref{thm:limit-lambda-white}.}

\begin{proof}
From the upper bound in Theorem \ref{white:upperbound}, we have that for any $x\in \rd$
$$
\E|u_t(x)|^2\leq c_1e^{c_2\lambda^{\frac{2\alpha}{\alpha-d\beta}}t}\quad \text{for\,all}\quad t>0,
$$
from which we have
$$
\limsup_{\lambda\to\infty} \frac{\log\log \E|u_t(x)|^2}{\log \lambda}\leq \frac{2\alpha}{\alpha-d\beta}.
$$
Next, we will establish a lower bound. Fix $x\in \rd$, for any $t>0$, we can always find a time $t_0$ such that $t=t-t_0+t_0$ and $t-t_0>0$.  If $t$ is already large enough so that $x\in B(0,t^{\beta/\alpha})$ then by Proposition \ref{whitenoiselowbd} and Lemma \ref{lemma:exp-lower-bound} we get
$$
\liminf_{\lambda\to\infty} \frac{\log\log \E|u_t(x)|^2}{\log \lambda}\geq \frac{2\alpha}{\alpha-d\beta}.
$$

Now if $x\notin B(0,t^{\beta/\alpha})$, we can choose a $\kappa>0$ so that $x\in B(0,(\kappa t)^{\beta/\alpha})$.  Then we can use the ideas in Proposition \ref{whitenoiselowbd} to end up with
\begin{equation*}
\E|u_{t+t_0}(x)|^2\geq g^2_{\kappa t}\sum_{k=0}^\infty \left(\lambda^2l_\sigma^2c_1 \right)^k\left(\frac{t}{k}\right)^{k(\alpha-\beta d)/\alpha},
\end{equation*}
and the result follows from this using Lemma \ref{lemma:exp-lower-bound}.
\end{proof}
\section{Proofs for the colored noise case.}
\subsection{ Proof of upper bound in Theorem \ref{thm-colored-noise}.}
\begin{proof}
The proof of existence and uniqueness is standard.  For more information, see\cite{walsh}.  We set
\begin{equation*}
u^{(0)}(t,\,x):=(\sG u_0)_t(x),
\end{equation*}
and
\begin{equation*}
u^{(n+1)}(t,\,x):=(\sG u_0)_t(x)+\lambda\int_0^t\int_{\R^d}G_{t-s}(x-y)\sigma(u^{(n)}(s,\,y))F(\d y\,\d s), \quad n\geq 0.
\end{equation*}
Define $D_n(t\,,x):=\E|u^{(n+1)}(t,\,x)-u^{(n)}(t,\,x)|^2$,  $H_n(t):=\sup_{x\in \R^d}D_n(t\,,x)$ and $\Sigma(t,y,n)=|\sigma(u^{(n)}(t,\,y))-\sigma(u^{(n-1)}(t,\,y))|$. We will prove the result for $t\in [0,\,T]$ where $T$ is some fixed number. We now use this notation together with the covariance formula \eqref{covariance-colored} and the assumption on $\sigma$ to write
\begin{equation*}
\begin{aligned}
&D_n(t,\,x)  \\
&=\lambda^2\int_0^t\int_{\R^d}\int_\rd G_{t-s}(x-y)G_{t-s}(x-z)\E[\Sigma(s,y,n) \Sigma(s,z,n)] f(y,z)\d y d z \d s.
\end{aligned}
\end{equation*}
Now we estimate the expectation on the right hand side using Cauchy-Schwartz inequality.
\begin{equation*}
\begin{aligned}
\E[\Sigma(s,y,n) \Sigma(s,z,n)] &\leq L_\sigma^2\E|u^{(n)}(s,\,y)-u^{(n-1)}(s,\,y)||u^{(n)}(s,\,z)-u^{(n-1)}(s,\,z)|\\
&\leq L_\sigma^2\bigg(\E|u^{(n)}(s,\,y)-u^{(n-1)}(s,\,y)|^2\bigg)^{1/2}\\
&\ \ \ \ \bigg(\E|u^{(n)}(s,\,z)-u^{(n-1)}(s,\,z)| ^2\bigg)^{1/2}\\
&\leq L_\sigma^2 \bigg( D_{n-1}(s,y) D_{n-1}(s,z)\bigg)^{1/2}\\
&\leq L_\sigma^2 H_{n-1}(s).
\end{aligned}
\end{equation*}
Hence we have for $\gamma<\alpha$ using  Lemma  \ref{lemma:covariance-upper-bound}
\begin{equation*}
\begin{aligned}
&D_n(t,\,x)  \\
&\leq \lambda^2L_\sigma^2\int_0^t H_{n-1}(s)\int_{\R^d}\int_\rd G_{t-s}(x-y)G_{t-s}(x-z) f(y,z)\d y \d z \,\d s\\
&\leq c_1 \lambda^2L_\sigma^2\int _0^t \frac{H_{n-1}(s)}{(t-s)^{\gamma \beta/\alpha}}\,\d s.
\end{aligned}
\end{equation*}
We therefore have
\begin{equation*}
H_{n}(t)\leq c_1 \lambda^2L_\sigma^2\int _0^t \frac{H_{n-1}(s)}{(t-s)^{\gamma \beta/\alpha}}\,\d s.
\end{equation*}
We now note that the integral appearing on the right hand side of the above display is finite when $d<\alpha/\beta$. Hence, by Lemma 3.3 in Walsh \cite{walsh}, the series $\sum_{n=0}^\infty H^{\frac12}_n(t)$ converges uniformly on $[0,\,T].$ Therefore, the sequence $\{u_n\}$ converges  in $L^2$ and uniformly on $[0,\,T]\times\R^d$ and the limit satisfies (\ref{mild-sol-colored}).  We can prove uniqueness in a similar way.

We now turn to the proof of the exponential bound. Set
$$
A(t):=\sup_{x\in\rd}\E|u_t(x)|^2.
$$
We claim that
there exists constants $c_4, c_5$ such that  for all $t>0, $ we have
$$
A(t)\leq c_4+ c_5(\lambda L_\sigma)^2\int_0^t \frac{A(s)}{(t-s)^{\beta \gamma/\alpha}}\,\d s.
$$
The renewal inequality in Proposition \ref{prop:renewal-upper} with $\rho=(\alpha-\gamma\beta)/\alpha$  then proves the exponential upper bound. To prove this claim, we start with the mild formulation given by \eqref{mild-sol-colored}, then take the second moment to obtain the following
\begin{equation}
\begin{split}
&\E|u_t(x)|^2=|(\mathcal{G}u)_t(x)|^2\\
&\ \ \ \ +\lambda^2 \int_0^t\int_{\rd\times\rd}G_{t-s}(x,y)G_{t-s}(x,z)f(y,z)\E[\sigma(u_s(y))\sigma(u_s(z))]\d y\d z\d s\\
&\ \ \ \ =I_1+I_2.
\end{split}
\end{equation}

Since $u_0$ is bounded, we have $I_1\leq c_4$. Next we use the assumption on $\sigma$ together with H\"older's inequality to see that
\begin{equation}
\begin{split}
\E[\sigma(u_s(y))\sigma(u_s(z))]&\leq L_\sigma^2 \E[u_s(y)u_s(z)]\\
&\leq [\E|u_s(y)|^2]^{1/2} [\E|u_s(z)|^2]^{1/2}\\
&\leq\sup_{x\in \rd}\E|u_s(x)|^2.
\end{split}
\end{equation}
Therefore, using Lemma \ref{lemma:covariance-upper-bound} the second term $I_2$ is thus bounded as follows.
$$
I_2\leq c_5 (\lambda L_\sigma)^2\int_0^t\frac{A(s)}{(t-s)^{\beta\gamma/\alpha}}\, \d s.
$$
Combining the above estimates, we obtain the required result in the claim.

\end{proof}
\subsection{ Proof of lower bound in Theorem \ref{thm-colored-noise}.}
The starting point of the proof of the lower bound hinges on the following recursive argument.
\begin{equation*}
\begin{aligned}
\E|u_t&(x)|^2\\
&=|(\sG u)_t(x)|^2+\lambda^2 \int_0^{t}\int_{\R^d\times\R^d}\\
& G(t-s_1,\,x,\,z_1)G(t-s_1,\,x,\,z_1')\E[\sigma(u_{s_1}(z_1))\sigma(u_{s_1}(z_1'))f(z_1,z_1')]\d z_1\d z_1'\d s_1.
\end{aligned}
\end{equation*}
We now use the assumption that $\sigma(x)\geq l_\sigma|x|$ for all $x$ to reduce the above to
 \begin{equation*}
\begin{aligned}
\E|u_t&(x)|^2\\
&\geq|(\sG u)_t(x)|^2+\lambda^2 l_\sigma^2\int_0^{t}\int_{\R^d\times\R^d}\\
& G(t-s_1,\,x,\,z_1)G(t-s_1,\,x,\,z_1')\E|u_{s_1}(z_1)u_{s_1}(z_1')|f(z_1,z_1')\d z_1\d z_1'\d s_1.
\end{aligned}
\end{equation*}
We now replace  the $t$ above by $t+\tilde{t}$  and use a substitution to reduce the above to

\begin{equation*}
\begin{aligned}
\E|u_{t+\tilde{t}}&(x)|^2\\
&\geq|(\sG u)_{t+\tilde{t}}(x)|^2+\lambda^2 l_\sigma^2\int_0^{t}\int_{\R^d\times\R^d}\\
& G(t-s_1,\,x,\,z_1)G(t-s_1,\,x,\,z_1')\E|u_{\tilde{t}+s_1}(z_1)u_{\tilde{t}+s_1}(z_1')|f(z_1,z_1')\d z_1\d z_1'\d s_1.
\end{aligned}
\end{equation*}
We also have
\begin{equation*}
\begin{aligned}
&\E|u_{\tilde{t}+s_1}(z_1)u_{\tilde{t}+s_1}(z_1')|\geq|(\sG u)_{\tilde{t}+s_1}(z_1)(\sG u)_{\tilde{t}+s_1}(z_1)|+\lambda^2l_\sigma^2\int_0^{s_1}\int_{\R^d\times\R^d} \\
&G(s_1-s_2,\,z_1,\,z_2)G(s_1-s_2,\,z_1',\,z_2')\E|u_{\tilde{t}+s_2}(z_2)u_{\tilde{t}+s_2}(z_2')|f(z_2, z_2')\d z_2\d z_2'\d s_2.
\end{aligned}
\end{equation*}
The above two inequalities thus give us
\begin{equation}
\begin{aligned}
\E|u_{t+\tilde{t}}&(x)|^2\\
&\geq|(\sG u)_{t+\tilde{t}}(x)|^2+\lambda^2 l_\sigma^2\int_0^{t}\int_{\R^d\times\R^d}\\
& G(t-s_1,\,x,\,z_1)G(t-s_1,\,x,\,z_1')\E|u_{\tilde{t}+s_1}(z_1)u_{\tilde{t}+s_1}(z_1')|f(z_1,z_1')\d z_1\d z_1'\d s_1\\
&\geq |(\sG u)_{\tilde{t}+t}(x)|^2+\lambda^2l_\sigma^2\int_0^{t}\int_{\R^d\times\R^d}\\
&G(t-s_1,\,x,\,z_1)G(t-s_1,\,x,\,z_1')f(z_1,z_1')(\sG u)_{\tilde{t}+s_1}(z_1)(\sG u)_{\tilde{t}+s_1}(z_1')\d z_1\d z_1'\d s_1\\
&+(\lambda l_\sigma)^4\int_0^{t}\int_{\R^d\times\R^d}G(t-s_1,\,x,\,z_1)G(t-s_1,\,x,\,z_1')f(z_1,z_1')\int_0^{\tilde{t}+s_1}\int_{\R^d\times\R^d}\\
&G(s_1-s_2,\,z_1,\,z_2)G(s_1-s_2,\,z_1,'\,z_2')
\E|u_{\tilde{t}+s_2}(z_2)u_{\tilde{t}+s_2}(z_2')|f(z_2, z_2')\d z_2\d z_2'\d s_2\d z_1\d z_1'\d s_1.
\end{aligned}
\end{equation}

We set $z_0=z_0':=x$ and $s_0:=t$ and continue the recursion as above to obtain
\begin{equation}\label{recursion}
\begin{aligned}
\E|u_{\tilde{t}+t}&(x)|^2\\
&\geq|(\sG u)_{\tilde{t}+t}(x)|^2\\
&+ \sum_{k=1}^\infty (\lambda l_\sigma)^{2k}\int_0^t\int_{\R^d\times\R^d}\int_0^{s_1}\int_{\R^d\times\R^d}\cdots \int_0^{s_{k-1}}\int_{\R^d\times\R^d} |(\sG u)_{\tilde{t}+s_k}(z_k)(\sG u)_{\tilde{t}+s_k}(z_k')|\\
&\prod_{i=1}^kG(s_{i-1}-s_{i}, z_{i-1},\,z_i)G(s_{i-1}-s_{i}, z'_{i-1},\,z'_i)f(z_i, z_i') \d z_i\d z_i'\d s_i.
\end{aligned}
\end{equation}
Therefore,
\begin{equation*}
\begin{aligned}
\E|u_{\tilde{t}+t}&(x)|^2\\
&\geq|(\sG u)_{\tilde{t}+t}(x)|^2\\
&+ \sum_{k=1}^\infty (\lambda l_\sigma)^{2k}\int_0^t\int_{\R^d\times\R^d}\int_0^{s_1}\int_{\R^d\times\R^d}\cdots \int_0^{s_{k-1}}\int_{\R^d\times\R^d} |(\sG u)_{\tilde{t}+s_k}(z_k)(\sG u)_{\tilde{t}+s_k}(z_k')|\\
&\prod_{i=1}^kG(s_{i-1}-s_{i}, z_{i-1},\,z_i)G(s_{i-1}-s_{i}, z'_{i-1},\,z'_i)f(z_i, z_i') \d z_i\d z_i'\d s_i.
\end{aligned}
\end{equation*}

\begin{proposition}\label{colorednoiselowbd}
There exists a $t_0>0$ such that for  $t>t_0$,
\begin{equation*}
\E|u(t+t_0,\,x)|^2\geq g_t^2\sum_{k=0}^\infty \left(\lambda^2l_\sigma^2c_1 \right)^k\left(\frac{t}{k}\right)^{k(\alpha-\gamma \beta )/\alpha}\quad\text{whenever}\ x \in B(0, t^{\beta/\alpha}),
\end{equation*}
where $c_1$ is a positive constant.
\end{proposition}

\begin{proof}
We will look at the following term which comes from the recursive relation described above,
\begin{equation*}
\begin{aligned}
\sum_{k=1}^\infty& (\lambda l_\sigma)^{2k}\int_0^t\int_{\R^d\times\R^d}\int_0^{s_1}\int_{\R^d\times\R^d}\cdots \int_0^{s_{k-1}}\int_{\R^d\times\R^d} |(\sG u)_{\tilde{t}+s_k}(z_k)(\sG u)_{\tilde{t}+s_k}(z_k')|\\
&\prod_{i=1}^kG(s_{i-1}-s_{i}, z_{i-1},\,z_i)G(s_{i-1}-s_{i}, z'_{i-1},\,z'_i)f(z_i, z_i') \d z_i\d z_i'\d s_i.
\end{aligned}
\end{equation*}
We can bound the above term by
\begin{equation*}
\begin{aligned}
g_t^2\sum_{k=1}^\infty& (\lambda l_\sigma)^{2k}\int_0^t\int_{\R^d\times\R^d}\int_0^{s_1}\int_{\R^d\times\R^d}\cdots \int_0^{s_{k-1}}\int_{B(0,\,t^{\beta/\alpha})\times B(0,\,t^{\beta/\alpha})} \\
&\prod_{i=1}^kG(s_{i-1}-s_{i}, z_{i-1},\,z_i)G(s_{i-1}-s_{i}, z'_{i-1},\,z'_i)f(z_i, z_i') \d z_i\d z_i'\d s_i.
\end{aligned}
\end{equation*}
We now make a substitution and reduce the temporal region of integration to write
\begin{equation*}
\begin{aligned}
g_t^2\sum_{k=1}^\infty& (\lambda l_\sigma)^{2k}\int_0^{t/k}\int_{\R^d\times\R^d}\int_0^{t/k}\int_{\R^d\times\R^d}\cdots \int_0^{t/k}\int_{B(0,\,t^{\beta/\alpha})\times B(0,\,t^{\beta/\alpha})} \\
&\prod_{i=1}^kG(s_{i}, z_{i-1},\,z_i)G(s_{i}, z'_{i-1},\,z'_i)f(z_i, z_i') \d z_i\d z_i'\d s_i.
\end{aligned}
\end{equation*}

We will further reduce the domain of integration so the function
\begin{equation*}
\prod_{i=1}^kG(s_{i}, z_{i-1},\,z_i)G(s_{i}, z'_{i-1},\,z'_i)f(z_i, z_i'),
\end{equation*}
has the required lower bound.  For $i=0,\cdots, k$, we set

\begin{equation*}
z_i\in B(x,\,s_1^{\beta/\alpha}/2)\cap B(z_{i-1},\,s_i^{\beta/\alpha})
\end{equation*}
and
\begin{equation*}
z'_i\in B(x,\,s_1^{\beta/\alpha}/2)\cap B(z'_{i-1},\,s_i^{\beta/\alpha}).
\end{equation*}
We therefore have $|z_i-z'_i|\leq s_1^{\beta/\alpha}$, $|z_i-z_{i-1}|\leq s_i^{\beta/\alpha}$ and $|z'_i-z'_{i-1}|\leq s_i^{\beta/\alpha}$. We use the lower bound on the heat kernel to find that
\begin{equation*}
\begin{aligned}
\prod_{i=1}^k&G(s_{i}, z_{i-1},\,z_i)G(s_{i}, z'_{i-1},\,z'_i)f(z_i, z_i')\\
&\geq \frac{c^k}{s_1^{k\gamma\beta/\alpha}}\prod_{i=1}^k\frac{1}{s_i^{2\beta d/\alpha}},
\end{aligned}
\end{equation*}
for some $c>0$.
We set $\sA_i:=B(x,\,s_1^{\beta/\alpha}/2)\cap B(z_{i-1},\,s_i^{\beta/\alpha})$ and $\sA_i':=B(x,\,s_1^{\beta/\alpha}/2)\cap B(z'_{i-1},\,s_i^{\beta/\alpha})$.
We will further choose that $s_i^{\beta/\alpha}\leq \frac{s_1^{\beta/\alpha}}{2}$ and note that $|\sA_i|\geq c_1s_i^{d\beta/\alpha}$ and $|\sA'_i|\geq c_1s_i^{d\beta/\alpha}$. We therefore have

\begin{equation*}
\begin{aligned}
g_t^2\sum_{k=1}^\infty& (\lambda l_\sigma)^{2k}\int_0^{t/k}\int_{\R^d\times\R^d}\int_0^{t/k}\int_{\R^d\times\R^d}\cdots \int_0^{t/k}\int_{B(0,\,t^{\beta/\alpha})\times B(0,\,t^{\beta/\alpha})} \\
 &\prod_{i=1}^kG(s_{i}, z_{i-1},\,z_i)G(s_{i}, z'_{i-1},\,z'_i)f(z_i, z_i') \d z_i\d z_i'\d s_i\\
 &\geq g_t^2\sum_{k=1}^\infty (\lambda l_\sigma)^{2k}\int_0^{t/k}\int_{\sA_1\times\sA'_1}\int_0^{s_1/2^{\beta/\alpha}}\int_{\sA_2\times\sA'_2}\cdots \int_0^{s_1/2^{\beta/\alpha}}\int_{\sA_k\times\sA'_k}\\
&\frac{1}{s_1^{k\gamma\beta/\alpha}}\prod_{i=1}^k\frac{1}{s_i^{2\beta d/\alpha}}\d z_i\d z_i'\d s_i\\
&\geq g_t^2\sum_{k=1}^\infty(\lambda l_\sigma c_2)^{2k} \int_0^{t/k}\frac{1}{s_1^{k\gamma \beta/\alpha}}s_1^{k-1}\d s_1\\
&\geq g_t^2\sum_{k=1}^\infty(\lambda l_\sigma c_3)^{2k} \left(\frac{t}{k}\right)^{k(1-\gamma\beta/\alpha)}.
\end{aligned}
\end{equation*}

We now take time large enough and use Lemma \ref{lemma:exp-lower-bound} to complete the proof of theorem.
\end{proof}
\subsection{Proof of Theorem \ref{excitation-colored}.}
The proof of this theorem is exactly as that of Theorem \ref{thm:limit-lambda-white} and it is omitted.
\qed
\subsection{Proof of Theorem \ref{thm:holder-exponent}.}
\begin{proof}
We will make use of the Kolmogorov's continuity theorem.
Therefore we consider the increment $\E|u_{t+h}(x)-u_{t}(x)|^p$ for $h\in (0,1)$ and $p\geq 2$. We have
\begin{equation}
\begin{split}
u_{t+h}(x)-u_{t}(x)&=\int_\rd[G_{t+h}(x-y)-G_{t}(x-y)]u_0(y)\d y\\
&+\lambda \int_0^t\int_\rd  [G_{t+h-s}(x-y)- G_{t-s}(x-y)]\sigma(u_s(y))F(\d s\, \d y)\\
&+\lambda \int_t^{t+h}\int_\rd  G_{t+h-s}(x-y) \sigma(u_s(y))F(\d s\, \d y).
\end{split}
\end{equation}
The first term, $\int_\rd G_{t}(x-y)u_0(y)\d y$ is  smooth for $t>0$. This essentially follows from the fact that under the condition on the initial condition, we can interchange integral and derivatives.  We will therefore look at higher moments of the remaining terms. Recall that  we can use the similar ideas in the proof of Theorem \ref{thm-colored-noise} to show that $\sup_{x\in \rd}\E|u_t(x)|^p$ is finite for all $t>0$. We use the Burkholder's inequality together with Proposition \ref{prop:temporal-increment-bound} to write

\begin{equation}
\begin{split}
&\E|\int_0^t\int_\rd  [G_{t+h-s}(x-y) G_{t-s}(x-y)]\sigma(u_s(y))F(\d s\d y)|^p \\
&\leq  c_1|\int_0^t\int_{\rd}|G^\ast_{t-s+h}(\xi)-G^\ast_{t-s}(\xi)|^2\frac{1}{|\xi|^{d-\gamma}}\d \xi \d s |^{p/2}\\
&\leq c_1 h^{\frac{p(1-\beta\gamma/\a)}{2}}.
\end{split}
\end{equation}

Similarly we have

\begin{equation}
\begin{split}
&\E|\int_t^{t+h}\int_\rd  G_{t+h-s}(x-y) \sigma(u_s(y))F(\d s\d y)|^p \\
&\leq c_2|\int_t^{t+h}\int_\rd \int_\rd G_{t+h-s}(x-y) G_{t+h-s}(x-z)f(y,z)\d s\d y\d z|^{p/2}\\
& \leq c_2|\int_t^{t+h}\int_\rd  [E_\beta(-\nu|\xi|^\alpha( t+h-s)^\beta)]^2 \frac{1}{|\xi|^{d-\gamma}}\d \xi \d s|^{p/2}\\
&\leq c_1 (\frac{C_1^*h}{1-\beta\gamma/\a})^{\frac{p(1-\beta\gamma/\a)}{2}}.
\end{split}
\end{equation}
Combine the above estimates, we see that
$$
\E|u_{t+h}(x)-u_{t}(x)|^p\leq C h^{\frac{p(1-\beta\gamma/\a)}{2}}.
$$
Now an application of Kolmogorov's continuity theorem as in \cite{boulanba-et-al-2010} completes the proof.
\end{proof}

\end{document}